\documentclass{amsart}

\usepackage{mathrsfs}
\usepackage{hyperref}
\usepackage{enumerate}
\usepackage{bbm}

\hypersetup{
	colorlinks=false,
	pdfborder={0 0 0},
	bookmarksnumbered = true,
	bookmarksopen = true
}

\newtheorem{theorem}{Theorem}[section]
\newtheorem{lemma}[theorem]{Lemma}
\newtheorem{proposition}[theorem]{Proposition}
\newtheorem{corollary}[theorem]{Corollary}
\theoremstyle{definition}
\newtheorem{definition}[theorem]{Definition}

\newtheorem{remark}[theorem]{Remark}
\theoremstyle{remark}

\newtheorem{case2}{Case}

\numberwithin{equation}{section}

\newcommand{\abs}[1]{\left| \hspace{1pt} #1 \hspace{1pt} \right|}
\newcommand{\norm}[1]{\| \hspace{1pt} #1 \hspace{1pt} \|}
\newcommand{\spa}[1]{\mathrm{span}\left( #1\right)}

\newcommand{\close}[1]{\overline{#1}}

\newcommand{\inj}{\mathop{i_g\hrs{1}}}

\newcommand{\conv}{\mathop{c_g\hrs{1}}}
\newcommand{\sconv}{\mathop{sc_g\hrs{1}}}
\newcommand{\lconv}{\mathop{lc_g\hrs{1}}}
\newcommand{\slconv}{\mathop{slc_g\hrs{1}}}
\newcommand{\dist}[2]{{d_g}({#1},{#2})}

\newcommand{\oball}[2]{{B_g}\left({#2}\hrs{1};{#1}\right)}
\newcommand{\cball}[2]{\overline{B_g\left({#2}\hrs{1};{#1}\right)}}
\newcommand{\sfer}[2]{{S_g}\left({#2}\hrs{1};{#1}\right)}
\newcommand{\foball}[1]{\mathscr{B}_g ({#1})}
\newcommand{\fcball}[1]{\overline{\mathscr{B}}_g({#1})}

\newcommand{\hrs}[1]{\hspace{#1pt}}
\newcommand{\vrs}[1]{\vspace{#1pt}}

\newcommand{\ctype}[1]{{\bf {#1}}}

\newcommand{\cond}[1]{\ifnum#1=1 $\mathcal{C}_{\bf psl}$\else $\mathcal{C}_{\bf ps}$\fi}

\newcommand{\real}[1]{\mathbb{R}^{\hrs{0}#1}}
\newcommand{\nat}[1]{\mathbb{N}^{\hrs{0}#1}}

\title[Convexity Types]{Geodesic Convexity Types in Riemannian Manifolds}

\author{Octavian Mitrea}
\address{Department of Mathematics, University of Western Ontario, London, Ontario, Canada, N6A 5B7}
\email{omitrea@uwo.ca}

\begin{document}

\begin{abstract}
The usual notion of set-convexity, valid in the classical Euclidean context, metamorphoses into several distinct convexity types in the more general Riemannian setting. By studying this phenomenon in reverse, we characterize complete manifolds for which certain convexity types are assumed a priori to coincide.
\end{abstract}

\maketitle

\section{Introduction} \label{sec:intro}

Classical convexity of sets in Euclidean space is an intuitive, simple to define property: a set is {\it convex} if it includes every line segment having its endpoints in that subset. The definition of the analog concept in the more general Riemannian context, where the notion of line segment is replaced by that of {\it geodesic segment}, is less straightforward, the situation being somewhat aggravated by the potential existence of more than one geodesic joining two distinct points or, even by the absence of connecting segments, in the case of non-complete metrics. The solution consists in introducing a separate kind of convexity criterion for each relevant scenario thus, producing a variety of convexity types characterizing subsets of the manifold: {\it (proper) convexity}, {\it weak convexity}, {\it total convexity}, {\it strong convexity}, etc.  One can see the above as diversifying one mathematical concept into several related but more specialized notions, their distinctive features addressing the requirements imposed by the more general setting.

One could also look at the same process from a reversed angle and give it the following converse interpretation: when working with subsets of the Euclidean space, many of the different convexity types operating in the Riemannian setting, loose their distinctiveness, and "blend" their stand-alone features back into the classical concept of set-convexity. This seemingly moot change of perspective suggests, in fact, that we may start from the assumption that certain convexity types "coincide", in a sense that will be made precise in this paper, and analyze the impact such assumption may have on the geometry of the manifold. One of the questions we raise and analyze by adopting this approach, is to what extent, such fusion of convexity notions is characteristic to the Euclidean space and, whether there are other geometrically or topologically distinct spaces that accommodate similar phenomena. Specifically, we show in Theorem \ref{th:main} that the only complete Riemannian manifolds in which the proper (usual) and strong convexity types coincide (denoted as \cond{2} manifolds) are those of infinite convexity radius. Combining this with Proposition 4 in \cite{Sul74} we obtain the main result of the paper,
\begin{theorem} \label{th:main0}
A complete manifold $M$ is \cond{2} if and only if $M$ is simply connected and focal-point-free. 
\end{theorem}

The paper is structured as follows. Section \ref{sec:prelim} provides a high-level account of key prerequisites. We include here Berger's classical inequality \cite{Ber76}, $\displaystyle \conv(M) \leq \frac{\inj(M)}{2}$, for which we give an independently devised proof, based on a known decomposition of the cut locus due to Bishop \cite{Bi77}. Using a minimum amount of model-theoretic techniques, in Section \ref{sec:convequiv} we define in more formal terms the notions of {\it convexity type} and {\it type equivalence}. The definitions are specifically tailored to the scope and goals set for this article, with the sole intention of setting the grounds for a more rigorous presentation, reference and notation mechanism. To prove Theorem \ref{th:main} we first analyze a slightly different scenario, that where the proper, strong and a version of local convexity (subsection \ref{ssec:localconvexity}) types coincide (\cond{1} manifolds). In Section \ref{sec:main1} we show that all such manifolds have, in fact, an infinite (strong) convexity radius. Using this fact, in Section \ref{sec:main2} we prove Theorem \ref{th:main} and implicitly, Theorem \ref{th:main0}. 

I wish to thank the anonymous referee for the most useful criticism on an earlier version of the paper, especially for outlining the steps towards obtaining the final proof of Theorem \ref{th:main}.

\subsection{Notations and Conventions} \label{ssec:notations}
Throughout the paper, $(M, g)$ denotes a connected Riemannian manifold of dimension $n \geq 2$, endowed with a complete metric $g$. By a {\it geodesic segment} we mean a length minimizing geodesic defined on a closed interval. For $x, y \in M$, $\gamma_{xy} : [a, b] \rightarrow M$ denotes the geodesic segment, if it exists and it is unique (up to parameterization), starting at $x$ and ending at $y$. We may write $\gamma_x$ to indicate the geodesic segment $\gamma_{px}$, when $p \in M$ is known in the context. We use the name {\it segment} (as opposed to "geodesic segment") to refer to the image of $\gamma_{xy}$ in $M$: $[xy] = \gamma_{xy}( [a, b] )$. An {\it open segment} is defined as $(xy) = \gamma_{xy}((a, b))$. We will be specific about the geodesics' parameterizations every time the analysis requires it. For a geodesic parameterized by arc length we will often use the term {\it normal geodesic}. All other conventions and notations are part of the standard: $T_pM$ and $TM$ represent the manifold's tangent space at $p \in M$ and its tangent bundle, respectively. The space of smooth sections of $TM$ is denoted by $\mathscr{T}(M)$ maintaining the same notation, $\mathscr{T}(N)$, for its restriction to a subset $N \subset M$. For a subspace $H \subset T_pM$, $\bot H$ indicates the orthogonal complement of $H$ in $T_pM$. If $v \in T_pM$ we write $\bot v$ instead of $\bot \spa{v}$, where $\spa{v}$ indicates the subspace of $T_pM$ generated by $v$. For an immersed submanifold $N \subset M$, we use $\bot_pN$, instead of $\bot T_pN$, to denote the normal space of $N$ at $p \in N$, while $\bot N$ indicates the normal bundle of $N$ in $M$. For $v \in TN$, $w \in \bot N$, $S_w v$ denotes the shape operator of $N$. The open (closed) geodesic balls of radius $r$ centered at $p \in M$ are denoted by $\oball{r}{p}$ ($\cball{r}{p}$) while $\oball{r}{0_p}$ ($\cball{r}{0_p}$) represent open (closed) balls centered at the origin in $T_pM$. Similar notations $\sfer{r}{p}$, $\sfer{r}{0_p}$, apply to the spheres in $M$, $T_pM$ respectively. The connection being considered will always be the (Levi-Civita) connection associated to the metric. The Einstein summation convention will be used throughout the material, whenever feasible.

\section{Preliminaries} \label{sec:prelim}

In this section we briefly review the notions and techniques we employ in this paper.

\begin{definition} \label{def:convex}
Let $(M, g)$ be a complete Riemannian manifold of dimension $n \geq 1$ and $C$ a subset of $M$. 
\begin{enumerate}[$(a)$]
\item \label{it:convex1} $C$ is called {\it convex} if for all $x, y \in C$ there is a unique (up to parameterization) geodesic segment $\gamma_{xy}$ in $M$ and $[xy] \subset C$;
\item \label{it:convex2} $C$ is said to be {\it strongly convex} if for all $x, y \in \close{C}$ (the closure of $C$ in $M$) there is a unique (up to parameterization) geodesic segment $\gamma_{xy}$ in $M$ and $(xy) \subset C$;
\end{enumerate}
\end{definition}

\begin{remark} \label{rem:properconv}
To distinguish it from strong (or other forms of) convexity we may occasionally refer to a convex set as being {\it properly convex}, or to the notion itself as {\it proper convexity}. We may also adopt the same  approach to distinguish local convexity from strong local convexity (see subsection \ref{ssec:localconvexity} below).
\end{remark}

\subsection{A Version of Local Convexity} \label{ssec:localconvexity}
We will also make use of a local version of convexity, defined for geodesic balls, as presented in \cite{BC64}. Denoting by $\dist{\cdot}{\cdot}$ the distance function induced by $g$ on $M$, we first define two distinct properties a geodesic sphere of $M$ may satisfy, adopting a naming convention similar to the one in \cite[p.246]{BC64}:

\begin{definition} \label{def:cc} We say that $\sfer{r}{p} \subset M$ satisfies the {\it convexity condition} (or briefly, it satisfies {\it c.c.}) if for every $x \in \sfer{r}{p}$ and for every geodesic $\gamma$ tangent to $\sfer{r}{p}$ at $x = \gamma(t_0)$, there is $ \varepsilon > 0$ such that $\dist{p}{\gamma(t)} \geq r$, for all $t \in (t_0 - \varepsilon, t_0+ \varepsilon)$. If we require that $\dist{p}{\gamma(t)} > r$ for all $t \in (t_0 - \varepsilon, t_0) \cup (t_0, t_0+ \varepsilon)$, then we say that $\sfer{r}{p}$ satisfies the {\it strong convexity condition} ({\it s.c.c.}).
\end{definition}

\begin{definition} \label{def:localconv1}
A geodesic ball $\oball{r}{p}$ is {\it locally convex} if $\sfer{s}{p}$ satisfies c.c, for all $0<s < r$. Replacing c.c. with s.c.c. we obtain the definition of a {\it strongly locally convex} ball.
\end{definition}

Local convexity is closely related to smooth geodesic variations and associated $N$-Jacobi fields (for more details on this classic concept see for example \cite[pp. 221-224]{BC64}). Let $N$ be an immersed submanifold of $M$ and $\gamma : [0, r] \rightarrow M$, a normal geodesic segment with its starting point in $N$ and perpendicular to $N$. 

\begin{definition} \label{def:njacobi}
An {\it $N$-Jacobi field} $J$ is a Jacobi field along $\gamma$ satisfying:
\begin{enumerate}[\;\;\;(i)]
\item \label{it:njacob1} $J(0) \in T_{\gamma(0)}N$;
\item \label{it:njacob2} $\displaystyle \langle J(t), \dot{\gamma}(t) \rangle = 0$, $\forall t \in [0, r]$;
\item \label{it:njacob3} $\displaystyle S_{\dot{\gamma}(0)}J(0) - D_tJ(0) \in \bot_{\gamma(0)}N$.
\end{enumerate}
\end{definition}

If $N = \{p\}$ is a single point in $M$ we call $J$ a {\it $p$-Jacobi field} and the conditions above reduce to $J(0) = 0$ and $\displaystyle \langle D_tJ(0) , \dot{\gamma}(0) \rangle = 0$. The set of all $N$-Jacobi fields, which we denote by $\mathscr{J}_N(\gamma)$, is a subspace of dimension $n-1$ of the vector space of all Jacobi fields along $\gamma$. Let $I:\mathscr{T}(\gamma) \rightarrow \real{}$ be the quadratic form associated with the index form of $\gamma$. If $\gamma$ has its endpoint $\gamma(r)$ in another immersed submanifold $P$, being also orthogonal to $P$, the restriction of $I$ to $\mathscr{J}_N(\gamma)$ has the form \cite[p.220]{BC64}:
\begin{equation} \label{eq:indexform}
I(J) = \big \langle S_{\dot{\gamma}(0)}J(0) -D_tJ(0), J(0) \big \rangle - \big \langle S_{\dot{\gamma}(r)}J(r) -D_tJ(r), J(r) \big \rangle,
\end{equation}
where we wrote $\langle \cdot , \cdot \rangle$ instead of $g(\cdot,\cdot)$ and the first and second shape operators refer to $N$ and $P$, respectively. The formula follows directly from the definition of the index form and the Jacobi equation.

Our analysis relies in part on a basic construction which can be found, for example, in \cite[p.247]{BC64}, \cite[p.97]{Cha06}. Let $p, x \in M$ such that $r = \dist{p}{x}$ satisfies $0 < r < \inj(p)$, where $\inj(p)$ is the injectivity radius of $p$. Let $\gamma_x : [0, r] \rightarrow M$ be the unique normal geodesic segment starting at $p = \gamma_x(0)$ and ending at $x = \gamma_x(r)$. If $V_x \subset M$ is a neighborhood of $x$, define $N_x = \exp_{x}(\bot\dot{\gamma}_x(r)) \cap V_x$, where $\bot\dot{\gamma}_x(r)$ denotes the orthogonal complement in $T_xM$ of $\spa{\dot{\gamma}_x(r)}$. It can be easily seen that, for a small enough $V_x$, $N_x$ is  an immersed submanifold of codimension $1$ which contains all geodesic segments at $x$ with images in $V_x$, that are orthogonal to $\dot{\gamma}_x(r)$. Denoting by $I(x, J)$, $J \in \mathscr{J}_p(\gamma_x)$, the quadratic form (\ref{eq:indexform}) associated to $\gamma_x$, we get: 
\begin{equation} \label{eq:indexformpjacob}
I(x, J) = \big \langle D_tJ(r), J(r) \big \rangle,
\end{equation}
since for all $J \in \mathscr{J}_p(\gamma_x)$, both shape operators vanish and $J$ vanishes at $0$. The following lemma, relates the definiteness of $I$ with the property of a geodesic sphere satisfying the convexity conditions.

\begin{lemma} \label{lem:ccindexform}
Let $\sfer{r}{p}$ be included in a normal neighborhood of $p$. If $\sfer{r}{p}$ satisfies c.c. then, for all $x \in \sfer{r}{p}$, $I(x, \cdot)$ is positive semi-definite on $\mathscr{J}_p(\gamma_x)$. $\sfer{r}{p}$ satisfies s.c.c if and only if $I(x, \cdot)$ is positive definite on $\mathscr{J}_p(\gamma_x)$, $\forall x \in \sfer{r}{p}$.
\end{lemma}

The proof is straightforward and it can be found, for example, in \cite[p.247]{BC64}.

In Lemma \ref{lem:ccindexform} we assumed that $\oball{r}{p}$ is included in a normal neighborhood of $p$. In fact, as pointed out in \cite[p.246]{BC64}, in a complete space, the local convexity of $\oball{r}{p}$ requires $\exp_p$ to be a diffeomorphism on $\oball{r}{0_p} \in T_pM$. We also have (\cite[p.246]{BC64}):
\begin{lemma} \label{lem:convlocalconv}
If $\oball{r}{p}$ is convex (strongly convex) then $\sfer{r}{p}$ satisfies c.c.. (s.c.c.). In particular, every (strongly) convex open ball is (strongly) locally convex.
\end{lemma}

\subsection{Convexity and Injectivity Radii} \label{ssec:coninjrad}

All of the convexity types considered in this paper carry significant upper bounds on the radii of geodesic balls, as follows:

\begin{definition} \label{def:convradii}
For each point $p \in M$ we define:
\begin{enumerate}[$(a)$]
\item \label{it:convrad} The {\it (proper) convexity radius} at $p$,
\begin{equation*}
\conv(p) = \sup \hrs{1}\{\;\rho>0 : \oball{r}{p} \;\hbox{is convex}, \; \forall \hrs{1} 0<r<\rho \};
\end{equation*}
\item \label{it:sconvrad} The {\it strong convexity radius} at $p$,
\begin{equation*}
\sconv(p) = \sup \hrs{1}\{\;\rho>0 : \oball{r}{p} \;\hbox{is strongly convex}, \;\forall \hrs{1} 0<r<\rho\};
\end{equation*}
\item \label{it:lconvrad} The {\it (proper) local convexity radius} at $p$,
\begin{equation*}
\lconv(p) = \sup \hrs{1}\{\;\rho>0 : \; \oball{\rho}{p} \;\hbox{is locally convex}\};
\end{equation*}
\item \label{it:slconvrad} The {\it strong local convexity radius} at $p$,
\begin{equation*}
\slconv(p) = \sup \hrs{1}\{\;\rho>0 : \; \oball{\rho}{p} \;\hbox{is strongly locally convex}\};
\end{equation*}
\end{enumerate}
\end{definition}
The manifold's respective radii are defined as $\conv(M) = \inf \hrs{1} \{\hrs{1} \conv(p) : p \in M \hrs{1}\}$, with the analog formulas for $\sconv(M)$, $\lconv(M)$ and $\slconv(M)$. It is shown (see for example \cite{BC64}, \cite{DoC92}, \cite{Whi32}) that all of the convexity radii in Definition \ref{def:convradii} are positive at each point of the manifold, whereas their global counterparts ($\conv(M)$, etc.) may be zero in some non-compact scenarios. Of course all radii can be infinite, as it is the case in $\real{n}$.

Recall that, in a complete manifold $(M, g)$, the injectivity radius, $\inj(p)$, at a point $p$ is defined as the radius of the largest ball centered at $p$ on which the exponential map is a diffeomorphism. Equivalently, by the Morse index theorem, it equals the distance from $p$ to its cut locus, $C_p$. The manifold's injectivity radius, $\inj(M)$, is then given as the greatest lower bound of all $\inj(p)$, $p \in M$. It is known that $\inj(\cdot)$ is continuous on $M$, $C_p$ is closed and, for $C_p \neq \emptyset$, $\inj(p) = \dist{p}{C_p}$ is always realized by some point of $C_p$ (see for example \cite{CE75}, \cite{Sak96}).

Next, we derive a necessary and sufficient condition for an open subset  $A \subset M$ to be uniquely geodesic, by which we mean that every pair of points in $A$ are joined by exactly one geodesic segment. Subsequently, we shall utilize this lemma to give a simple and, we believe, original proof of a classical inequality concerning the global convexity and injectivity radii in a complete manifold due to Berger \cite{Ber76}. The proof of the lemma relies on a known classification of points in the cut locus due to Bishop, which we briefly review next.

\begin{remark}[Bishop's Theorem on the Decomposition of the Cut Locus] \label{rem:cutlocus}
Under a complete metric on $M$, the points in $C_p$ can be categorized in two classes: {\it ordinary} cut points, connected to $p$ by two or more geodesic segments, and {\it singular} cut points $q$, connected to $p$ by exactly one minimizing geodesic segment along which $q$ is first conjugate to $p$ (note that an ordinary cut point for $p$ can also be conjugate to $p$ along some of the connecting geodesics). It is shown in \cite{Bi77} that {\it the set of ordinary cut points is dense in $C_p$}, which implies that the singular cut points are accumulation points for the ordinary ones in $C_p$.
\end{remark}

\begin{lemma} \label{lem:uniquegeodesic}
Let $A \subset M$ be an open subset of a complete manifold $(M, g)$. $A$ is uniquely geodesic if and only if it does not contain pairs of points that are cut points to each other. In particular, every open properly or strongly convex set has the above property.
\end{lemma}
\begin{proof}
If $A$ is uniquely geodesic then, clearly, $x, y \in A$ can not be ordinary cut points. If they are singular, then by Bishop's theorem, there is a sequence $y_n \in C_{x}$, $n \in \nat{}$, of ordinary cut points of $x$, converging to $y$. But, since $A$ is open, we would get an infinity of ordinary points $y_n$ inside $S$, which is a contradiction. The converse inference is trivial. The last statement follows from the fact that all properly or strongly convex sets are uniquely geodesic.
\end{proof}

The following global result, due to Berger \cite{Ber76}, gives a (sharp) inequality between the radii, $\conv(M)$ and $\inj(M)$. Although in \cite{Ber76} only the compact case is considered, the formula is also valid for non-compact complete manifolds. We show that the result is a straightforward consequence of Bishop's theorem.

\begin{proposition}[Berger, 1976] \label{prp:convinj_2}
In a complete manifold $(M, g)$, we have
\begin{equation} \label{eq:convinj_2}
\conv(M) \leq \frac{\inj(M)}{2}.
\end{equation}
\end{proposition}
\begin{proof}
If $\conv(M) = \infty$ then $\inj(M) = \infty$, since it is known that $\conv(p) \leq \inj(p)$ for all $p \in M$, and so $(\ref{eq:convinj_2})$ is satisfied trivially. Otherwise, assume $\conv(M) > \inj(M)/2$. If $\inj(M) = \inj(x)$ for some $x \in M$, then we continue the analysis using $\oball{\inj(x)}{x}$. If $\inj(M)$ is not realized, which could be possible if $M$ is not compact, then we can find $x \in M$ such that $\inj(x)/2 < \conv(M)$. Either way, choose $p$ as the midpoint of a geodesic segment connecting $x$ to one of its cut points $y$ on $\sfer{\inj(x)}{x}$. It follows that $x, y \in \oball{\conv(p)}{p}$, which is convex hence, $x, y$ cannot be cut points to each other, hence a contradiction.
\end{proof}

The next lemma is known in the local convex case {\it (\ref{it:localtoconvex1})}. In {\it (\ref{it:localtoconvex2})} we address the strong local convexity scenario, which is slightly different from the former and  which we could not find a reference for.

\begin{lemma} \label{lem:localtoconvex}
Let $(M, g)$ be complete, of dimension $n \geq 2$, $p \in M$ and $r > 0$. Then,
\begin{enumerate}[{\it (i)}]
\item \label{it:localtoconvex1} If $\oball{r}{p}$ is locally convex, then every segment with endpoints in $\oball{r/2}{p}$ is entirely in $\oball{r/2}{p}$;
\item  \label{it:localtoconvex2} If $\oball{r}{p}$ is strongly locally convex then, for all $0<r'<r/2$, every segment with endpoints in $\cball{r'}{p}$ has its interior contained in $\oball{r'}{p}$.
\end{enumerate}
\end{lemma}
\begin{proof}
We only show {\it (\ref{it:localtoconvex2})}; for a proof of {\it (\ref{it:localtoconvex1})} see, for example, \cite[p.246]{BC64}. Let $0<r'<r/2$ and let $\gamma : [0, 1] \rightarrow M$ be a geodesic segment with $x=\gamma(0) \in \cball{r'}{p}$ and $y=\gamma(1) \in \cball{r'}{p}$. We have $\dist{x}{y} \leq \dist{p}{x} + \dist{p}{y} \leq 2r' < r$. Assuming that there is $z = \gamma(t_0)$ such that $\dist{p}{z} \geq r$, by the triangle inequality we have $\dist{x}{z} \geq \dist{p}{z} - \dist{p}{x} \geq r-r' > r - r/2 = r/2$ and, similarly, $\dist{y}{z} > r/2$, hence $\dist{x}{y} = \dist{x}{z} + \dist{y}{z} > r$, which is absurd. It follows that $[xy]$ is entirely contained in $\oball{r}{p}$. Put $\rho = \max\{ \dist{p}{x}, \dist{p}{y}\}$ so, $\rho \leq r'$. We notice that $\dist{p}{\gamma(\cdot)}$ cannot have any maximum values, for if  $w = \gamma(t_1)$, $0<t_1<1$ were such maximum, then $\gamma$ would be tangent at $w$ to $\sfer{\dist{p}{w}}{p}$, and, for a small enough neighborhood of $t_1$, $\gamma$ would be contained in $\oball{\dist{p}{w}}{p}$, i.e. $\sfer{\dist{p}{w}}{p}$ would not satisfy s.c.c.. Since $\sfer{\dist{p}{w}}{p} \subset \oball{r}{p}$, this would be in contradiction with the strong local convexity of $\oball{r}{p}$. It follows, then, that $\gamma$ does not leave $\cball{\rho}{p}$. Finally, the possibility of a portion of $(xy)$ lying on $\partial \oball{\rho}{p}$ is again contrary to the strong local convexity of $\oball{r}{p}$, since in that case $\sfer{\rho}{p}$ would satisfy c.c. but not s.c.c.. Therefore, we must have $(xy) \subset \oball{\rho}{p} \subseteq \oball{r'}{p}$, which completes the proof.
\end{proof}

The following relations comparing the radii corresponding to different types of convexity, will be of assistance further in the sequel. Some of the formulas are immediate and all are known (or essentially known) to the community, in one form or another.

\begin{proposition} \label{prp:convrad} Let $(M, g)$ be a complete Riemannian manifold.
\begin{enumerate}[(i)]
\item \label{it:convrad1} For all points $p \in M$, we have:
\begin{equation} \label{eq:convrad11}
\sconv(p) \leq \conv(p), \hrs{5} \slconv(p) \leq \lconv(p);
\end{equation}
\begin{equation}  \label{eq:convrad12}
\conv(p) \leq \lconv(p), \hrs{5} \sconv(p) \leq \slconv(p);
\end{equation}
\begin{equation}  \label{eq:convrad13}
\lconv(p) \leq \inj(p);
\end{equation}

\item \label{it:convrad2} In addition to the global counterparts of (\ref{eq:convrad11}), (\ref{eq:convrad12}) and (\ref{eq:convrad13}), the following relations also hold:
\vrs{3}
\begin{equation} \label{eq:convrad22}
\frac{\lconv(M)}{2} \leq \conv(M);
\end{equation}
\begin{equation} \label{eq:convrad23}
\frac{\slconv(M)}{2} \leq \sconv(M).
\end{equation}
\end{enumerate}
\end{proposition}
\begin{proof}
{\it (\ref{it:convrad1})} Relations (\ref{eq:convrad11}) and (\ref{eq:convrad12}) follow immediately from Definition \ref{def:convradii} and Lemma \ref{lem:convlocalconv}, respectively and for a proof of (\ref{eq:convrad13}), see for example \cite[p.246]{BC64}.

{\it (\ref{it:convrad2})} The global equivalents of (\ref{eq:convrad11}), (\ref{eq:convrad12}) and (\ref{eq:convrad13}) are immediate. To prove (\ref{eq:convrad22}) and assuming the contrary, we can find $p \in M$ and a real $r_0$, such that $0<r_0<\lconv(M) - 2\conv(p)$. By (\ref{eq:convrad13}) we have $\lconv(M) \leq \inj(M)$, which implies that, for all $x \in \cball{\conv(p) + r_0/2}{p}$ we have $\oball{\conv(p) + r_0/2}{p} \subset \oball{\inj(M)}{x}$, i.e. $\oball{\conv(p) + r_0/2}{p}$ is totally normal. This means that any two points in $\oball{\conv(p) + r_0/2}{p}$, are joined by exactly one geodesic segment. On the other hand, all balls centered at $p$ of radii not greater than $\lconv(p)/2$ satisfy the conclusion of Lemma \ref{lem:localtoconvex}{\it (\ref{it:localtoconvex1})} and, since $\conv(p) + r_0/2 < \lconv(M)/2 \leq \lconv(p)/2$, it follows that so do all balls of radii less than $\conv(p) + r_0/2$. But this means that for all $0<r<r_0$, the open balls $\oball{\conv(p) + r/2}{p}$ are convex, which is in contradiction with the definition of $\conv(p)$. The proof of  (\ref{eq:convrad23}) follows similarly from Lemma \ref{lem:localtoconvex}{\it (\ref{it:localtoconvex2})}.
\end{proof}

\begin{corollary} \label{cor:inftyconvradii}
In a complete, connected manifold $(M, g)$ of dimension $n \geq 2$, we have:
\begin{enumerate}[{\it (i)}]
\item \label{it:inftyconvradii1} $\lconv(M) = \infty \Leftrightarrow \conv(M) = \infty$;
\item \label{it:inftyconvradii2} $\slconv(M) = \infty \Leftrightarrow \sconv(M) = \infty$.
\end{enumerate}
\end{corollary}
\begin{proof}
Both points are direct consequences of (\ref{eq:convrad12}), (\ref{eq:convrad22}) and (\ref{eq:convrad23}).
\end{proof}

\section{Equivalence of Convexity Types} \label{sec:convequiv}

Let $M$ be a set and $\mathscr{P}(M)$ the collection of all subsets of $M$. Let $\mathcal{R}_2(M)$ be the collection of equivalence relations on $\mathscr{P}(M)$ having exactly two classes of equivalence.

\begin{definition} \label{def:prop}
A {\it property of subsets of} $M$ (or just a {\it property on $M$}) is a relation $\rho \in \mathcal{R}_2(M)$.
\end{definition}

\noindent {\bf Convention:} If $\rho \in \mathcal{R}_2(M)$ is a property on $M$, we denote by $\widehat{0}_\rho$ and $\widehat{1}_\rho$ its two equivalence classes. Once the notation is in place, we refer to an element of $\widehat{1}_\rho$ as an element {\it having property $\rho$}, with the obvious consequence on the naming of the elements of $\widehat{0}_\rho$, as {\it not} having property $\rho$.

\begin{remark}
In practice, the above convention does not impose any rules on which of the two equivalence classes should be denoted by $\widehat{0}_\rho$ or $\widehat{1}_\rho$, offering no guarantees that the members of $\widehat{1}_\rho$ do indeed have the corresponding property. For a meaningful usage throughout the paper, we shall assign the notations so that they make sense in the context being analyzed.
\end{remark}

\begin{definition} \label{def:propequiv}
Two properties, $\rho_1, \rho_2 \in \mathcal{R}_2(M)$, are said to be {\it equivalent} with respect to a collection of subsets $\mathcal{A} \subseteq \mathscr{P}(M)$, written $\rho_1 \sim_{\mathcal{A}} \rho_2$, if $A \in \widehat{1}_{\rho_1} \Leftrightarrow A \in \widehat{1}_{\rho_2}$, for all $A \in \mathcal{A}$. 
\end{definition}

Clearly, $\sim_{\mathcal{A}}$ is an equivalence relation on $\mathcal{R}_2(M)$. Next, let us consider $(M, g)$ a complete Riemannian manifold of dimension $n \geq 2$. We define the following equivalence relations on $\mathscr{P}(M)$ by giving explicitly their equivalence classes. In fact, since for each relation there will be two classes in total, we will indicate only one of them, the other being the complementary of the former with respect to $\mathscr{P}(M)$:

\renewcommand{\labelitemi}{$\diamond$}

\begin{itemize}
\item $\ctype{P} \in \mathcal{R}_2(M)$, with $\widehat{1}_\ctype{P} = \{ A \subseteq M : A \hbox{ is properly convex} \}$;
\item $\ctype{S} \in \mathcal{R}_2(M)$, with $\widehat{1}_\ctype{S} = \{ A \subseteq M : A \hbox{ is strongly convex}\}$;
\item $\ctype{L} \in \mathcal{R}_2(M)$, with $\widehat{1}_\ctype{L} = \{ A \subseteq M : A \hbox{ is locally convex}\}$; 
\item $\ctype{SL} \in \mathcal{R}_2(M)$, with $\widehat{1}_\ctype{SL} = \{ A \subseteq M : A \hbox{ is strongly locally convex} \}$.
\end{itemize}

Evidently, all four relations defined above are properties on $M$ in the sense of Definition \ref{def:prop}. Let $\mathcal{C} = \{\ctype{P}, \ctype{S}, \ctype{L}, \ctype{SL} \} \subset \mathcal{R}_2(M)$ be the set of the four properties defined above.

\begin{definition} \label{def:convequiv}
An element of $\mathcal{C}$ is referred to as a {\it convexity type}. We say that two convexity types are {\it equivalent} ({\it coincide}, or are {\it indistinguishable}) with respect to a given family of subsets of $M$, if they are equivalent as properties, in the sense of Definition \ref{def:propequiv}.
\end{definition}

The rest of the paper is dedicated to the study of scenarios where some of the convexity types in $\mathcal{C}$ are equivalent. Specifically, the first such case will assume the coincidence of $\ctype{P}$, $\ctype{S}$ and $\ctype{SL}$, while the second one will weaken the condition by imposing only the equivalence of $\ctype{P}$ and $\ctype{S}$. All equivalences are considered over families of geodesic balls in $M$. On that note, before ending the section, let us introduce the following notations, which will be used in the sequel:

\renewcommand{\labelitemi}{}
\begin{itemize}
\item $\foball{p} = \{ \oball{r}{p} :  0< r \leq \inj(p) \}$; \vrs{5}
\item $\fcball{p} = \{ \cball{r}{p} : 0< r < \inj(p) \}$; \vrs{5}
\item $\displaystyle \foball{M} = \bigcup_{p \in M} \foball{p}$, \hrs{2} $\displaystyle \fcball{M} = \bigcup_{p \in M} \fcball{p}$.
\end{itemize}

\section{The \ctype{P}\ctype{S}\ctype{L} Equivalence}  \label{sec:main1}

\begin{definition} \label{def:pslequiv}
We say that $p \in M$ {\it satisfies condition} \cond{1} if the \ctype{P}, \ctype{S} and \ctype{SL} convexity types coincide on $\foball{p}$: 
\begin{equation*}
\ctype{P} \sim_{\foball{p}} \ctype{S} \sim_{\foball{p}} \ctype{SL}.
\end{equation*}
The metric $g$ is said to be a {\it \cond{1}-metric}, if every point of $M$ satisfies \cond{1}. A {\it \cond{1}-manifold} is a manifold endowed with a \cond{1}-metric. Equivalently, a \cond{1}-manifold satisfies $\ctype{P} \sim_{\foball{M}} \ctype{S} \sim_{\foball{M}} \ctype{SL}$.
\end{definition}

We require a few more preparatory facts. Let $p \in M$ and $0<r<\inj(p)$. The following lemma is an easy consequence of the classical theorem on the existence and uniqueness of Jacobi fields along a curve:

\begin{lemma} \label{lem:jacobiiso}
Let $p$, $r$ be as above and $v \in \sfer{r}{0_p}$. For $0< s <r$, let $\gamma_s : [0, s] \rightarrow M$ denote the normal geodesic segment given by $\gamma_s(0)=p$ and $\displaystyle \dot{\gamma}_s(0) = \frac{v}{r}$. Then, $\bot v$ and $\mathscr{J}_p(\gamma_s)$ are isomorphic.
\end{lemma}
\begin{proof}
Let $0<s <r$ and $\gamma_s$ be as given. The theorem of existence and uniqueness of Jacobi fields states that for any $ \widetilde{w} \in T_pM$ there is a unique Jacobi field $J \in \mathscr{J}(\gamma_s)$ satisfying $J(0) = 0_p$ and $D_tJ(0) =  \widetilde{w}$. Define $H^s_{v} : \bot v \rightarrow \mathscr{J}(\gamma_s)$, where, for all $ \widetilde{w} \in \bot v$, $H^s_{v}(\widetilde{w})$ is the unique Jacobi field corresponding to $ w$, as above. Thus, $H^s_{ v}$ is a bijection onto its image and the fact that $H^s_{ v}(\bot v) = \mathscr{J}_p(\gamma_s)$ follows immediately from the existence theorem. Linearity is a direct consequence of the linearity of the covariant derivative $D_t$.
\end{proof}

\begin{remark}
Choose normal coordinates at $p$ on $\cball{r}{p}$, and denote by $(\partial_i)$ the local coordinate frame. Using the known expression of a Jacobi field in normal coordinates (see for example \cite[p.178]{Lee97}), for all $t \in [0, s]$ we have:
\begin{equation} \label{eq:jacobiiso1}
H^s_{v}(\widetilde{w})(t) = t \widetilde{w}^i \partial_i  \big |_{\gamma_s(t)} \hrs{1},
\end{equation}
where $ \widetilde{w} =  \widetilde{w}^i \partial_i \big |_p$. Differentiating along $\gamma_s$ in ($\ref{eq:jacobiiso1}$), we get
\begin{equation} \label{eq:jacobiiso2}
D_t H^s_{v}(\widetilde{w})(t) =  \widetilde{w}^i \partial_i \big |_{\gamma_s(t)} + t \widetilde{w}^i D_t \partial_i \big |_{\gamma_s(t)}.
\end{equation}
We will use the coordinate expressions ($\ref{eq:jacobiiso1}$) and ($\ref{eq:jacobiiso2}$) in what follows.
\end{remark}

\noindent {\bf Spherical Coordinates.} The proof we chose for the next lemma, as well as for the main theorem, makes use of spherical coordinates on $T_pM$. To ensure the consistency of our presentation, we describe this setting in detail. Let $(e_i)$ be an orthonormal basis in $T_pM$ and denote by $E : T_pM \rightarrow \real{n}$ the natural isomorphism, $E( v^ie_i) = ( v^1,\; ... \;, v^{\hrs{1}n})$. Denote by $(\widetilde{\partial}_i)$ the global coordinate frame on $T_pM$ defined by the coordinate chart $(T_pM, E)$. If, for a tangent vector $ v \in T_pM$, we let $\iota_{ v} : T_pM \rightarrow T_{ v}T_pM$ be the natural identification, it is clear that $\iota_{ v}(e_i) = \widetilde{\partial}_i \big |_{ v}$, $\forall 1 \leq i \leq n$, hence, for all $ \widetilde{w} = \widetilde{ w}^{\hrs{1}i} e_i \in T_pM$ we have
\begin{equation} \label{eq:identcoords}
\iota_{ v}(\widetilde{w}) = \widetilde{ w}^{\hrs{1}i} \hrs{2} \widetilde{\partial}_i \big |_{ v}.
\end{equation}
Also, observe that $\iota_{ v}(\bot  v) = T_{ v}\hrs{1}\sfer{r}{0_p}$, for all $ v \in T_pM$, $\abs{v} = r$.

Consider spherical coordinates charts $\big\{(U_a,  \varphi_a)\big\}_a$ on $T_pM$, centered at $0_p$; here we have $U_a \subset T_pM$ with $\displaystyle \bigcup_a U_a = T_pM - 0_p$ and $ \varphi_a : U_a \rightarrow \real{n}$, $ \varphi_a( v) = (\theta^1( v),\; ... \;, \theta^{\hrs{1}n}( v))$ such that $0 < \theta^1( v) < 2\pi$, $0 < \theta^{\hrs{1}i}( v) < \pi$, $\forall \hrs{1} 2\leq i \leq n-1$ (if $n \geq 3$) and $\theta^{\hrs{1}n}( v) = \abs{ v}$. For every chart $(U_a,  \varphi_a)$, denote by $(\widetilde{\partial}^{\hrs{1}a}_i)$ the corresponding local coordinate frame on $U_a$. Then, for all $1\leq i\leq n$ and $ v \in U_a$, we have
\begin{equation} \label{eq:spherebasis}
\widetilde{\partial}_i^{\hrs{1}a} \big |_{ v} = \frac{\partial  v^j}{\partial \theta^{\hrs{1}i}} \Bigg |_{\hrs{1} \varphi_a( v)} \widetilde{\partial}_j \big |_{ v},
\end{equation}
where $( v^1,\; ... \;,  v^{\hrs{1}n}) = (E \circ  \varphi_a^{-1})(\theta^1,\; ... \;,\theta^{\hrs{1}n})$, is the associated transition map on $T_pM \cap U_a$, hence each component $v^{\hrs{1}i}$ is a smooth function\footnote{\hrs{2}The spherical coordinate transformations could, of course, be given explicitly. However, this is not required for the purposes set in this paper.} of $(\theta^1,\; ... \;,\theta^{\hrs{1}n})$.

Let $O_a = U_a \cap \sfer{r}{0_p}$ and define $\phi_a : O_a \rightarrow \real{n-1}$ given as $\phi_a( v) = (\theta^1( v), \;... \;, \theta^{\hrs{1}n-1}( v))$. Clearly, $\big \{(O_a, \phi_a) \big\}_a$ is a smooth atlas on $\sfer{r}{0_p}$ and $(\partial_1^{\hrs{1}a} (= \widetilde{\partial}_1^{\hrs{1}a}),\; ... \;, \partial_{n-1}^{\hrs{1}a} (=  \widetilde{\partial}_{n-1}^{\hrs{1}a}))$ is a smooth frame on $O_a$. The above atlas induces a natural smooth structure, $\big\{(V_a, \psi_a)\big\}_a$, on $T\sfer{r}{0_p}$, the sphere's tangent bundle, where $V_a = \{ ( v,  w) :  v \in O_a,  w \in T_{ v}\sfer{r}{0_p}  \}$ and $\psi_a( v,  w) = (\phi_a( v),  w^1,\; ... \;,  w^{\hrs{1}n-1})$, $\displaystyle  w = \sum_{i=1}^{n-1}  w^i \partial_i^{\hrs{1}a} \big |_{ v} \in T_{ v}\sfer{r}{0_p}$. From the fact that $\iota_{ v}(e_i) = \widetilde{\partial}_i \big |_{ v}$ and from (\ref{eq:spherebasis}), if $ w =  w^i \partial_i^{\hrs{1}a} \big |_{ v}$ (with summation from $1$ to $n-1$), we have
\begin{equation*}
\iota^{-1}_{ v}( w) =  w^i \iota^{-1}_{ v}(\partial_i^{\hrs{1}a} \big |_{ v}) =  w^i \frac{\partial  v^j}{\partial \theta^{\hrs{1}i}} \Bigg |_{\hrs{1} \varphi_a( v)} \iota^{-1}_{ v}(\widetilde{\partial}_j \big |_{ v}) =  w^i \frac{\partial  v^j}{\partial \theta^{\hrs{1}i}} \Bigg |_{\hrs{1} \varphi_a( v)} e_j,
\end{equation*}
and, making the notation $\displaystyle (\alpha_a)_i^j( v) = \frac{\partial  v^j}{\partial \theta^{\hrs{1}i}} \Bigg |_{\hrs{1} \varphi_a( v)}$, we get
\begin{equation} \label{eq:identinverse}
\iota^{-1}_{ v}( w) =  w^i (\alpha_a)_i^j( v) \hrs{1}e_j,
\end{equation}
with all $(\alpha_a)_i^j$ smooth maps on $O_a$.

For the rest of this section, the set $(0, r) \times T\sfer{r}{0_p}$ is considered with the product smooth structure, where the open real interval $(0, r)$ is endowed with the standard real smooth structure and $T\sfer{r}{0_p}$ with the structure described above. We shall preserve the notations in the above setting as well as those in Lemma \ref{lem:jacobiiso}. In addition, we denote by $\chi$ the normal coordinate map defined by $(e_i)$ and by $(\partial_i)$ the associated local coordinate frame on $\cball{r}{p}$.

\begin{lemma} \label{lem:gsmooth}
Let $p$, $r$ be as above. The map $G : (0, r) \times T\sfer{r}{0_p} \rightarrow \real{}$ given as $ G(t,  v,  w) = \Big \langle D_t H^t_{ v}(\iota^{-1}_{ v}(  w))(t), \;H^t_{ v}(\iota^{-1}_{ v}( w))(t) \Big \rangle$ is smooth on $(0, r) \times T\sfer{r}{0_p}$.
\end{lemma}
\begin{proof}
Let $\displaystyle \left( (0,r) \times V_a, \hrs{2} (\mathbbm{1}_{(0, r)}, \psi_a) \right)$ be a chart on $(0, r) \hrs{2}\times\hrs{2} T\sfer{r}{0_p}$, where $\mathbbm{1}_{(0, r)}$ denotes the identity map. We need to show that $G \circ (\mathbbm{1}_{(0, r)}, \psi_a)^{-1} : (0, r) \times \psi_a(V_a) \rightarrow \real{}$ is smooth. For $( v,  w) \in V_a$ we have $ v = \phi_a^{-1}(\theta^1,\; ... \;, \theta^{\hrs{1}n-1})$ and $ w =  w^i \partial_i^{\hrs{1}a} \big |_{ v}$, where $(\theta^1,\; ... \;, \theta^{\hrs{1}n-1}) \in (0, 2\pi)\times (0, \pi)^{n-2}$. Let $\gamma_v : [0, r] \rightarrow M$ be the normal geodesic satisfying $\gamma_v(0)=p$ and $\displaystyle \dot{\gamma}_v(0) = \frac{v}{r}$. Note that $\gamma_v(t) = \gamma_s(t)$, for all $0<t<s<r$, where $\gamma_s$ is defined as per Lemma \ref{lem:jacobiiso}. Taking in consideration this fact and writing in normal coordinates, for all $t \in (0, r)$, $( v,  w) \in V_a$, we have
\begin{equation} \label{eq:gsmooth1}
H^t_ v(\iota_{ v}^{-1}( w))(t) = t w^i \alpha_i^j( v)\; \partial_j \big |_{\gamma_{ v}(t)}
\end{equation}
and
\begin{equation} \label{eq:gsmooth2}
D_t H^t_{ v}( w)(t) =  w^i \alpha_i^j( v) \left( \partial_j \big |_{\gamma_{ v}(t)} + t D_t \partial_j \big |_{\gamma_{ v}(t)} \right),
\end{equation}
where we used (\ref{eq:jacobiiso1}), (\ref{eq:jacobiiso2}), (\ref{eq:identinverse}) and we wrote $\alpha_i^j$ instead of $(\alpha_a)_i^j$.
Expanding the covariant derivatives of the coordinate frame vectors, we get
\begin{equation} \label{eq:gsmooth3}
D_t \partial_j \big |_{\gamma_{ v}(t)} = \dot{\gamma}_{ v}^k(t) \Gamma_{kj}^l (\gamma_{ v}(t))\partial_l \big |_{\gamma_{ v}(t)}
\end{equation}
Replacing (\ref{eq:gsmooth1}), (\ref{eq:gsmooth2}) and (\ref{eq:gsmooth3}) in the expression of $G$, we obtain
\begin{equation*} 
\begin{split}
G(t,  v,  w) &= t\; w^i \alpha_i^j( v) \; w^k \alpha_k^l( v) \;\Big \langle \partial_j \big |_{\gamma_{ v}(t)}, \partial_l \big |_{\gamma_{ v}(t)} \Big \rangle \\& + 
t^2 \; w^i \alpha_i^j( v)\; w^k \alpha_k^l( v) \; \dot{\gamma}_{ v}^q(t) \hrs{1} \Gamma_{qj}^s \left(\gamma_{ v}(t) \right) \;\Big \langle \partial_s \big |_{\gamma_{ v}(t)}, \partial_l \big |_{\gamma_{ v}(t)} \Big \rangle \\&= t\; w^i \alpha_i^j( v) \; w^k \alpha_k^l( v) \left[  g_{jl} (\gamma_{ v}(t)) + t\; \dot{\gamma}_{ v}^q(t) \Gamma_{qj}^s (\gamma_{ v}(t))\; g_{sl} (\gamma_{ v}(t)) \right]
\end{split}
\end{equation*}
In normal coordinates, $\chi(\gamma_ v(t)) = ((t/r) v^1,\; ... \;,\; (t/r) v^n)$, and $\dot{\gamma}^i_{ v}(t) =  (t/r)v^i$, for all $1\leq i\leq n$, where $ v =  v^i e_i$. Putting 
\begin{equation*}
E( v) = (v^1,\; ... \;,\; v^n), 
\end{equation*}
where $E$ is the (global) coordinate map on $T_pM$ as above, the last equation becomes
\begin{equation} \label{eq:gsmooth5}
\begin{split}
G(t,  v,  w) &= t\hrs{1} w^i  w^k \hrs{1}\alpha_i^j( v) \hrs{1}\alpha_k^l( v) \hrs{1} \Big [  g_{jl} \left(\chi^{-1}\left(\frac{t}{r}E( v)\right) \right) \\ & + \frac{t}{r} v^q \hrs{1} \Gamma_{qj}^s \left(\chi^{-1}\left(\frac{t}{r}E( v)\right) \right)\; g_{sl} \left(\chi^{-1}\left(\frac{t}{r}E( v)\right) \right) \Big ].
\end{split}
\end{equation}
But, $ v = \phi_a^{-1}(\theta^1,\; ... \;, \theta^{\hrs{1}n-1})$ as well as each component $ v^q(\theta^1,\; ... \;, \theta^{\hrs{1}n-1})$ are smooth on $\phi_a(O_a)$ and, since all other intervening maps are smooth on their domains, it follows that $G \circ (\mathbbm{1}_{(0, r)}, \psi_a)^{-1}$ is smooth on $\displaystyle (0,r) \times \psi_a(V_a)$, which completes the proof.
\end{proof}

\begin{lemma} \label{lem:ccsmoothmap}
Let $p \in M$ and $0<r<\inj(p)$. Then, for all $t \in (0, r)$,
\begin{enumerate}[(i)]
\item \label{it:ccsmoothmap1}
If $\sfer{t}{p}$ satisfies c.c. then $G(t,  v,  w) \geq 0$, for all $( v,  w) \in T\sfer{r}{0_p}$;
\item \label{it:ccsmoothmap2}
$\sfer{t}{p}$ satisfies s.c.c. if and only if $G(t,  v,  w) > 0$, for all $( v,  w) \in T\sfer{r}{0_p}$.
\end{enumerate}
\end{lemma}
\begin{proof}
For $x \in \sfer{t}{p}$, let $\gamma_x : [0, t] \rightarrow M$ be the unique normal geodesic segment, connecting $p = \gamma_x(0)$ to $x = \gamma_x(t)$ and let $\displaystyle  v = r\dot{\gamma}_x(0)$. For $J \in \mathscr{J}_p(\gamma_x)$, let $ w = \iota_{v}((H^t_{v})^{-1}(J))$. It is clear that $(v,  w) \in T\sfer{r}{0_p}$ and, by construction, we have $I(x, J) = G(t,  v,  w)$. It is also immediate that $v$ and $w$ are uniquely determined by $x$ and $J$. 

Vice-versa, for $t \in  (0, r)$, $( v,  w) \in T\sfer{r}{0_p}$, put $x = \exp_p( \frac{t}{r} v)$ and let $\gamma_x : [0, t] \rightarrow M$ be the (unique) normal geodesic segment connecting $p$ and $x$. If $J = H^t_{ v}(\iota_{v}^{-1}(w))$, then $x \in \sfer{t}{p}$, $J \in \mathscr{J}_p(\gamma_x)$ are uniquely defined by $( v,  w)$ and $G(t,  v,  w) = I(x, J)$.

Using the above, both statements follow from Lemma \ref{lem:ccindexform}.
\end{proof}

\begin{corollary} \label{cor:locconvsmoothmap}
Let $p \in M$ and $0<r<\inj(p)$. Then:
\begin{enumerate}[(i)]
\item \label{it:localconv1}
If $\oball{r}{p}$ is locally convex then $G(t,  v,  w) \geq 0$ for all $0<t<r$ and $( v,  w) \in T\sfer{r}{0_p}$;
\item \label{it:localconv2}
$\oball{r}{p}$ is strongly locally convex if and only if $G(t,  v,  w) > 0$ for all $0<t<r$ and $( v,  w) \in T\sfer{r}{0_p}$.
\end{enumerate}
\end{corollary} 
\begin{proof}
The statements follow immediately from Lemma \ref{lem:ccsmoothmap} and Definition \ref{def:localconv1}.
\end{proof}

\begin{remark} \label{rem:convinjstrict}
As it was mentioned before, it is easy to see that $\conv(p) \leq \inj(p)$, $p \in M$. If $\inj(p) < \infty$ and the equivalence of convexity and strong convexity applies to all open balls centered at $p$ (which condition \cond{1} requires), then the inequality is strict:
\begin{equation}
\conv(p) < \inj(p). 
\end{equation}
To see this, notice first that, under the specified conditions, $\cball{\conv(p)}{p}$ is convex. Indeed, $\oball{\conv(p)}{p}$ is convex, hence strongly convex, which implies that its closure is (strongly) convex. On the other hand, since $\sfer{\inj(p)}{p}$ contains at least one cut point of $p$,  an analysis similar to that made in the proofs of the Proposition \ref{prp:convinj_2} would imply that, if $\inj(p) = \conv(p)$, then $\cball{\conv(p)}{p}$ can not be convex.
\end{remark}


\begin{theorem} \label{th:maintheorem1}
Let $(M, g)$ be a complete, connected Riemannian manifold of dimension $n \geq 2$. If $p \in M$ satisfies condition \cond{1}, then $\sconv(p) = \infty$. In particular, if there is such point in $M$, then $M$ is diffeomorphic to $\real{n}$.
\end{theorem}
\begin{proof}
Assume $\sconv(p) < \infty$. Condition \cond{1} implies that $\slconv(p) = \conv(p) = \sconv(p)$, quantity that we shall denote by $ \rho$ for the rest of the proof. By Remark \ref{rem:convinjstrict}, we have $0< \rho < \inj(p)$. Since no open ball of radius greater than $ \rho$ can be strongly locally convex, we can find a sequence of spheres of radii $ \rho < r_n <r$, $n \in \nat{}$, arbitrarily close to $ \rho$, such that the boundary of each ball in the sequence does not satisfy s.c.c.. According to Lemma \ref{lem:ccsmoothmap}, this means that, for a fixed $ \rho < r < \inj(p)$, there is a sequence $(r_n,  v_n,  w_n) \in (0, r) \times T\sfer{r}{0_p}$, $n \in \nat{}$, such that $G(r_n,  v_n,  w_n) \leq 0$ and $ \rho < r_n < r_{n+1} < r$. Clearly, $\displaystyle \lim_{n \rightarrow n}r_n =  \rho$. Since $\sfer{r}{0_p}$ is compact under the metric topology, there exists a tangent vector $ v \in \sfer{r}{0_p}$ and a subsequence $ v_{n_k}$, $k \in \nat{}$, such that $\displaystyle \lim_{k \rightarrow \infty}  v_{n_k} =  v$. Let $(O_a, \phi_a)$ be a spherical coordinate chart as described above, and $U$, an open set such that $ v \in U \subset \close{U} \subset O_a$. Since $ v_{n_k}$ converges to $ v$, we can assume, without loss of generality, that $ v_{n_k} \in U$, for all $k \in \nat{}$. To summarize, we obtained a sequence $(r_{n_k},  v_{n_k},  w_{n_k}) \in (0, r) \times T\sfer{r}{0_p}$, $k \in \nat{}$, satisfying the following properties:
\begin{equation} \label{eq:main11}
\lim_{k \rightarrow \infty} r_{n_k} =  \rho, \lim_{k \rightarrow \infty} v_{n_k} =  v,
\end{equation} 
\begin{equation*}
 v_{n_k},  v \in U \subset \close{U} \subset O_a, \forall k \in \nat{},
\end{equation*} 
\begin{equation*}
 w_{n_k} \in T_{ v_{n_k}}\sfer{r}{0_p} \; \mathrm{and} \;  w_{n_k} \neq 0_{ v_{n_k}}, \forall k \in \nat{},
\end{equation*} 
\begin{equation*}
G(r_{n_k},  v_{n_k},  w_{n_k}) \leq 0.
\end{equation*}

\noindent In coordinates $(V_a, \psi_a)$, we have:

\begin{equation*}
G \left(r_{n_k}, \hrs{1} \psi_a^{-1}(\theta_k^{\hrs{1}1},\; ... \;, \theta_k^{\hrs{1}n-1}, \hrs{1}  w_{n_k}^{\hrs{1}1},\; ... \;,  w_{n_k}^{\hrs{1}n-1} \hrs{1}) \right) \leq 0,
\end{equation*}

\noindent with $\theta_k^{\hrs{1}i} \in \phi_a(U)$ and $ w_{n_k}^{\hrs{1}i} \in \real{}$, for all $1 \leq i \leq n-1$, $k \in \nat{}$. Let $\displaystyle u_k = \frac{ w_{n_k}}{\norm{ w_{n_k}}}$, where $\displaystyle \norm{ w_{n_k}} = \left[ \sum_{i=1}^{n-1} ( w_{n_k}^i)^{\hrs{1}2} \right]^{1/2}$ is the Euclidean norm. From the explicit coordinate form of $G$ expressed in (\ref{eq:gsmooth5}), we have,
\begin{equation*}
G(r_{n_k},  v_{n_k},  u_k) = \frac{1}{\norm{ w_{n_k}}^2}\hrs{2}G(r_{n_k},  v_{n_k},  w_{n_k}),
\end{equation*}
hence,
\begin{equation*}
G \left(r_{n_k}, \hrs{1} \psi_a^{-1}(\theta_k^{\hrs{1}1},\; ... \;, \theta_k^{\hrs{1}n-1}, \hrs{1}  u_k^{\hrs{1}1},\; ... \;,  u_k^{\hrs{1}n-1} \hrs{1}) \right) \leq 0,
\end{equation*}
where $\displaystyle \left[ \sum_{i=1}^{n-1} ( u_k^i)^{\hrs{1}2} \right]^{1/2} = 1$. Let $r_1, r_2 \in \real{}$ be such that $0 < r_1 < \rho < r_2 < r$. Putting $K = [r_1, r_2] \times \phi_a(\close{U}) \times \sfer{1}{0_{\hrs{1}\real{n-1}}} \subset \real{2n-1}$, which is compact in $\real{2n-1}$, we can assume that, up to a finite number of terms, $b_k = (r_{n_k}, \theta_k^{\hrs{1}1},\; ... \;, \theta_k^{\hrs{1}n-1}, \hrs{1}  u_k^{\hrs{1}1},\; ... \;,  u_k^{\hrs{1}n-1} \hrs{1}) \in K$, for all $k \in \nat{}$. We can now extract a subsequence, denoted simply by $b_m$, $m \in \nat{}$, that converges to some $b = (r, \theta^{\hrs{1}1},\; ... \;, \theta^{\hrs{1}n-1}, \hrs{1}  u^{\hrs{1}1},\; ... \;,  u^{\hrs{1}n-1} \hrs{1}) \in K$. Recall that, by (\ref{eq:main11}), we have
\begin{equation*}
r = \lim_{m \rightarrow \infty} r_m = \rho \in [r_1, r_2]
\end{equation*}
and
\begin{equation*}
\lim_{m \rightarrow \infty} (\theta_m^{\hrs{1}1},\; ... \;, \theta_m^{\hrs{1}n-1}) = (\theta^{\hrs{1}1},\; ... \;, \theta^{\hrs{1}n-1}) = \phi_a(v) \in \phi(\close{U}).
\end{equation*}
But,
\begin{equation*}
(r_m, v_m, u_m) = (r_m, \psi_a^{-1}(\theta_m^{\hrs{1}1},\; ... \;, \theta_m^{\hrs{1}n-1}, \hrs{1}  u_m^{\hrs{1}1},\; ... \;,  u_m^{\hrs{1}n-1} \hrs{1})),
\end{equation*}
and, if we put $u = (u^{\hrs{1}1},\; ... \;,  u^{\hrs{1}n-1})$, by the continuity of $\psi_a^{-1}$, we get
\begin{equation*}
\lim_{m \rightarrow \infty} (r_m, v_m, u_m) = (\rho, v, u).
\end{equation*}
By construction, $G(r_m, v_m, u_m) \leq 0$, $\forall n \in \nat{}$, hence, by the continuity of $G$, we have $\displaystyle G(\rho, v, u) \leq 0$. On the other hand, $\oball{\rho}{p}$ is strongly convex which, by Lemma \ref{lem:convlocalconv}, implies that $\sfer{\rho}{p}$ satisfies s.c.c. and, by Lemma \ref{lem:ccsmoothmap}, $G(\rho, v, u) > 0$. The obvious contradiction indicates that our initial assumption is false, hence $\sconv(p) = \infty$. It follows that $\inj(p) = \infty$ which, since $M$ is connected, it means that $\exp_p$ is a diffeomorphism onto the entire manifold and this completes the proof of the theorem.
\end{proof}

\begin{corollary} \label{cor:main1}
A complete, connected manifold $(M, g)$ of dimension $n \geq 2$ is \cond{1} if and only if $\sconv(M) = \infty$.
\end{corollary}
\begin{proof}
It follows from Theorem \ref{th:maintheorem1}, Corollary \ref{cor:inftyconvradii} and Proposition \ref{prp:convrad}{\it (\ref{it:convrad1})}.
\end{proof}

\section{The \ctype{P}\ctype{S} Equivalence}  \label{sec:main2}

In this section we relax condition \cond{1} introduced in Section \ref{sec:main1}, by removing the strong local convexity from the list of constraints.

\begin{definition} \label{def:psequiv}
We say that $p \in M$ {\it satisfies condition} \cond{2} if the \ctype{P} and \ctype{S} convexity types coincide on $\foball{p}$: 
\begin{equation*}
\ctype{P} \sim_{\foball{p}} \ctype{S}.
\end{equation*}
The metric $g$ is said to be a {\it \cond{2}-metric}, if every point of $M$ satisfies \cond{2}. A {\it \cond{2}-manifold} is a manifold endowed with a \cond{2}-metric. Equivalently, a \cond{2}-manifold satisfies $\ctype{P} \sim_{\foball{M}} \ctype{S}$.
\end{definition}

\subsection{Main Result} The next lemma is a necessary and sufficient condition for a point of a complete manifold to satisfy \cond{2}.

\begin{lemma} \label{lem:main2}
Let $(M, g)$ be complete, connected, of dimension $n \geq 2$. Then, $p \in M$ satisfies condition \cond{2} if and only if either $\conv(p) =\sconv(p) = \infty$ or, the following properties hold for finite $\conv(p)$:
\begin{enumerate}[(i)]
\item \label{it:maintheorem2_0} $\cball{\conv(p)}{p}$ is the largest uniquely geodesic closed ball at $p$;
\item \label{it:maintheorem2_1} $\slconv(p) > \conv(p)$;
\end{enumerate}
\end{lemma}
\begin{proof}
If $\conv(p)=\infty$, since $p$ satisfying \cond{2} implies $\conv(p) = \sconv(p)$, we have $\sconv(p) = \infty$ and, conversely, if $\conv(p) = \sconv(p) = \infty$, then \cond{2} is satisfied trivially. Next, assume $\conv(p) < \infty$.

$(\Rightarrow)$ We show first {\it (\ref{it:maintheorem2_1})}. Clearly, $p$ satisfying \cond{2} implies that $\oball{\conv(p)}{p}$, along with all open balls at $p$ of lesser radii, are strongly convex hence, strongly locally convex (Lemma \ref{lem:convlocalconv}). Recall that by (\ref{eq:convrad12}) we have $\slconv(p) \geq \sconv(p)$. If $\slconv(p) = \sconv(p)$, it follows that no ball centered at $p$ of greater radius is strongly locally convex, thus $p$ satisfies \cond{1}. But, according to Theorem \ref{th:maintheorem1}, this is not possible for finite strong convexity radius at $p$.

To prove {\it (\ref{it:maintheorem2_0})}, observe again that, under condition \cond{1}, $\cball{\conv(p)}{p}$ is strongly convex, hence uniquely geodesic. To show that it is the largest such closed ball, we assume the contrary, which implies the existence of a positive real number $\varepsilon > 0$ such that $\oball{\conv(p)+\varepsilon}{p}$ is uniquely geodesic. By the definition of $\conv(p)$, there exists a sequence of non-convex open balls at $p$, of radii less than $\conv(p) +  \varepsilon$ which are strictly decreasing toward $\conv(p)$: $\oball{r_n}{p}$, $n \in \nat{}$, with $\conv(p) < r_{n+1} < r_n < \conv(p) +  \varepsilon$ and $\displaystyle \lim_{n \rightarrow \infty} r_n = \conv(p)$. Since all such balls are contained in a uniquely geodesic neighborhood, their non-convexity translates into the following property: for all $n \in \nat{}$, there are distinct points $x_n, y_n \in \oball{r_n}{p}$, such that $(x_ny_n) = \gamma_n((0, 1))$ is not included in $\oball{r_n}{p}$, where for every $n \in \nat{}$, $\gamma_n : [0, 1] \rightarrow M$ denotes the unique geodesic segment connecting $x_n$ to $y_n$. It follows that, for every $n \in \nat{}$, reducing $r_n$ by a small amount if necessary, the distance from $p$ to points of $[x_ny_n]$ attains a (possibly local) maximum at a point $z_n  \in (x_ny_n)$ outside the closed ball $\cball{r_n}{p}$. By the continuity of $\gamma_n$, there is a subsegment $[\widetilde{x}_n\widetilde{y}_n] \subset [x_ny_n]$ such that $z_n \in (\widetilde{x}_n\widetilde{y}_n)$, $(\widetilde{x}_n\widetilde{y}_n) \cap \cball{r_n}{p} = \emptyset$ and $\widetilde{x}_n, \widetilde{y}_n \in \sfer{r_n}{p}$. To avoid notational clutter, we reassign the notations $x_n$, $y_n$, $\gamma_n$ to refer to the points $\widetilde{x}_n$, $\widetilde{y}_n$ and the unique geodesic segment connecting them, respectively. Summarizing, we have identified sequences of points $x_n$, $y_n \in M$, $n \in \nat{}$, satisfying
\begin{equation} \label{eq:maintheorem24}
\dist{p}{x_n} = \dist{p}{y_n} = r_n,
\end{equation}
\begin{equation*} \label{eq:maintheorem25}
\dist{p}{z} - r_n > 0, \; \forall z \in (x_ny_n).
\end{equation*}

Due to the compactness of $\cball{\conv(p)+ \varepsilon}{p} \times \cball{\conv(p)+ \varepsilon}{p}$ and passing to a subsequence if necessary, we can assume $\{ (x_n, y_n) \}_n$ to converge to some $(x, y)$ in $\cball{\conv(p)+ \varepsilon}{p} \times \cball{\conv(p)+ \varepsilon}{p}$. We distinguish between two cases.

\begin{case2}[$x \neq y$]
From equation (\ref{eq:maintheorem24}), the fact that $\displaystyle \lim_{n \rightarrow \infty} r_n = \conv(p)$ and by the continuity of $\dist{p}{\cdot}$, it follows that $\displaystyle \dist{p}{x} = \dist{p}{y} = \conv(p)$, hence $x, y \in \sfer{\conv(p)}{p}$, which implies that $\{\gamma_n \}_{n \in \nat{}}$ converges (uniformly) to $\gamma$, where $\gamma : [0, 1] \rightarrow M$ is the geodesic segment connecting $x$ and $y$. Since $\cball{\conv(p)}{p}$ is strongly convex, we have $(xy) = \gamma((0, 1)) \subset \oball{\conv(p)}{p}$. On the other hand, the convergence of $\gamma_n$ to $\gamma$ implies that, up to a finite number of terms, $z_n \in \oball{\conv(p)}{p}$, which is in contradiction with the fact that $\dist{p}{z_n} > r_n > \conv(p)$.
\end{case2}

\begin{case2}[$x = y$] For better clarity in notations, put $z=x=y$, so that $\displaystyle \lim_{n \rightarrow \infty} x_n = \lim_{n \rightarrow \infty} y_n = z$. As we proceeded above, for each $n \in \nat{}$, let $z_n$ be a point on $(x_ny_n)$ realizing the maximum of $\dist{p}{\cdot}$ on $[x_ny_n]$. It follows that, for all $n \in \nat{}$, $\sfer{\dist{p}{z_n}}{p}$ does not satisfy s.c.c. since $\gamma_n$ is tangent to $\sfer{\dist{p}{z_n}}{p}$ at $z_n$ and $(x_ny_n) \setminus \{z_n\}$ is entirely contained in $\oball{\dist{p}{z_n}}{p}$. Due to the uniform convergence of $\gamma_n$ to the constant geodesic $\gamma : [0, 1] \rightarrow M$, $\gamma(t) = z$, we also have $\displaystyle \lim_{n \rightarrow \infty} z_n = z$.  Passing to the limit, we get
\begin{equation*} \label{eq:maintheorem28}
\lim_{n \rightarrow \infty} \dist{p}{z_n} = \dist{p}{z} = \conv(p),
\end{equation*}
which, along with the fact that, for all $n \in \nat{}$, $\sfer{\dist{p}{z_n}}{p}$ do not satisfy $s.c.c.$, implies that no open ball centered at $p$ of radius greater than $\conv(p)$ is strongly locally convex which is in contradiction with point {\it (\ref{it:maintheorem2_1})} of this theorem.
\end{case2}

$(\Leftarrow)$ We show first that $\oball{r}{p}$ is strongly convex, for all $0<r\leq \conv(p)$. Let $x, y \in \cball{r}{p}$, $x \neq y$. If both points are in the interior of the ball, then $(xy) \in \oball{r}{p}$ by the convexity of $\oball{r}{p}$ (note that for $r=\conv(p)$, the open ball $\oball{\conv(p)}{p}$ is also convex). If at least one of the points is on the boundary, then $(xy)$ does not leave $\cball{r}{p}$. For otherwise, there would be a $z \in (xy)$ such that $\dist{p}{z} > r$. Choose a sequence of geodesic segments with their endpoints in $\oball{r}{p}$ converging to $x$ and $y$, respectively. Since by {\it (\ref{it:maintheorem2_0})} the geodesic segment $\gamma_{xy}$ is unique, the sequence of geodesics would converge to $\gamma_{xy}$ which would imply that most of them contain points arbitrarily close to $z$, hence outside of $\cball{r}{p}$ and, thus, violating the convexity of $\oball{r}{p}$. Another possibility is to have at least a portion of $\gamma_{xy}$ lying on the boundary of $\oball{r}{p}$. In this case, $\gamma_{xy}$ would be tangent to $\sfer{r}{p}$, contradicting condition {\it (\ref{it:maintheorem2_1})}. So, all balls centered at $p$ of radii $r \leq \conv(p)$ are strongly convex. Since, by {\it (\ref{it:maintheorem2_0})}, any ball centered at $p$ and of radius greater than $\conv(p)$ is not uniquely geodesic, it follows that no such ball can be convex, therefore, all balls at $p$ are either strongly convex or not convex, i.e. $p$ satisfies \cond{2}, and this completes the proof.
\end{proof}

\begin{theorem} \label{th:main}
A complete Riemannian manifold is \cond{2} if and only if $\conv(M) = \infty$.  
\end{theorem}
\begin{proof}
One direction is clear so, for the converse, assume there exists a point $p \in M$ such that $\conv(p) <\infty$. By Remark \ref{rem:convinjstrict}, $\inj(p)-\conv(p) >0$. As in the proof of Lemma \ref{lem:main2}, there exists a sequence of non-convex open balls at $p$, $\oball{r_n}{p}$, $n \in \nat{}$, of radii less than $\inj(p)$, which are strictly decreasing toward $\conv(p)$. This implies the existence of a sequence of geodesic segments $\gamma_n : [0, 1] \rightarrow M$, whose images have endpoints in $\oball{r_n}{p}$ but are not entirely contained in $\oball{r_n}{p}$. As shown in the proof of Lemma \ref{lem:main2}, the distance function $t \mapsto \dist{p}{\gamma_n(t)}$ has a maximum at some $t_n \in(0,1)$ and, by Definition \ref{def:convradii}, we have
\begin{equation*} \label{eq:new1}
\dist{\gamma_n(t_n)}{p} \geq \slconv(p).
\end{equation*}
Taking the limit (up to a sequence, if necessary) and by Lemma \ref{lem:main2}(\ref{it:maintheorem2_1}), we obtain a geodesic segment $\gamma : [0, 1] \rightarrow M$ with endpoints in $\cball{\conv(p)}{p}$ and $\theta \in (0,1)$ such that 
\begin{equation} \label{eq:new2}
\dist{\gamma(\theta)}{p} \geq \slconv(p) > \conv(p).
\end{equation}

Let $x = \gamma(0)$ and $y = \gamma(1)$. For each $n \in \nat{}$, let $x_n, y_n \in \oball{r_n}{p}$ such that $x_n \rightarrow x$ and $y_n \rightarrow y$. For all $n \in \nat{}$, since $\oball{r_n}{p}$ is convex, there exists a geodesic segment $\alpha_n : [0,1] \rightarrow \oball{r_n}{p}$ connecting $x_n$ and $y_n$. Again, by taking the limit (up to a sequence) we obtain a geodesic segment $\alpha : [0,1] \rightarrow \cball{\conv(p)}{p}$ such that $\alpha(0)=x$ and $\alpha(1)=y$. This clearly implies that, for all $t \in [0,1]$,
\begin{equation} \label{eq:new3}
\dist{\alpha(t)}{p} \leq  \conv(p).
\end{equation}
By Lemma \ref{lem:main2}(\ref{it:maintheorem2_0}), $\cball{\conv(p)}{p}$ is uniquely geodesic, hence $\gamma = \alpha$, which cannot be true due to (\ref{eq:new2}) and (\ref{eq:new3}).
\end{proof}

As mentioned in the introduction, the above theorem together with Proposition 4 in \cite{Sul74} prove Theorem \ref{th:main0}.

\end{document}